\documentclass[12pt]{amsart}

\usepackage{enumerate}
\usepackage{verbatim}

\newtheorem{theorem}{Theorem}[section]
\newtheorem{corollary}[theorem]{Corollary}
\newtheorem{lemma}[theorem]{Lemma}
\newtheorem{proposition}[theorem]{Proposition}
\newtheorem{question}[theorem]{Question}

\theoremstyle{definition}

\theoremstyle{remark}

\numberwithin{equation}{section}

\newcommand{\eone}{Proposition \ref{NazInf} }

\newcommand{\bc}{\mathbb{C}}
\newcommand{\C}{\mathbb{C}}
\newcommand{\N}{\mathbb{N}}
\newcommand{\br}{\mathbb{R}}
\newcommand{\R}{\mathbb{R}}

\newcommand{\Z}{\mathbb{Z}}

\newcommand{\ve}{\varepsilon}

\newcommand{\bq}{\begin{equation}}
\newcommand{\eq}{\end{equation}}

\begin{document}

\title[Linear independence of time-frequency translates of functions with decay]
{Linear independence of time-frequency translates of functions with faster than exponential decay}

\author{Marcin Bownik}

\address{Department of Mathematics, University of Oregon, Eugene,
OR 97403--1222, USA}

\email{mbownik@uoregon.edu}

\author{Darrin Speegle}

\address{Department of Mathematics and Computer Science, Saint Louis 
University, 221 N. Grand Blvd.,
St. Louis, MO 63103, USA}

\email{speegled@slu.edu}

\keywords{Gabor system, exponential decay, HRT conjecture, linear independence, time-frequency translates, analytic zero divisor conjecture}


\subjclass[2000]{Primary: 26A99, 42C40, Secondary: 11K70}
\date{\today}

\begin{abstract} We establish the linear independence of time-frequency translates for functions $f$ having one sided decay $\lim_{x\to \infty} |f(x)| e^{cx \log x} = 0$ for all $c>0$. We also prove such results for functions with 
faster than exponential decay, i.e., $\lim_{x\to \infty} |f(x)| e^{cx} = 0$ for all $c>0$, under some additional restrictions.
\end{abstract}
\maketitle



\section{Introduction} 

The Heil-Ramanathan-Topiwala (HRT) conjecture \cite{HRT} states that time-frequency translates of a non-zero square integrable function $f$ on $\R$ are linearly independent. There have been a few partial results on the HRT conjecture.  Most results have focused on finding conditions on $\Lambda\subset \R^2$ which guarantee that time-frequency translates 
\[
{\mathcal {G}}(f,\Lambda) := \{M_aT_b f = e^{2\pi i a\cdot}f(\cdot - b): (a,b) \in \Lambda\}
\]
along $\Lambda$ are linearly independent. The most general type of partial result was obtained by Linnell \cite{Lin}, who showed that if $\Lambda$ is a lattice in $\R^2$ and $0 \ne f\in L^2(\R)$, then
${\mathcal {G}}(f,\Lambda)$ is linearly independent. The results in \cite{Lin} hold in the more general setting of $\R^d$, and they use the machinery of von Neumann algebras. Alternative proofs of Linnell's result were later given by the authors \cite{BS} and Demeter and Gautam \cite{DG}. Other interesting results have been obtained by Demeter and Zaharescu \cite{D, DZ} in the case that $\Lambda$ consists of four points which are contained in two parallel lines.  The HRT conjecture was also studied in the setting of locally compact abelian groups \cite{Ku} and for finite groups \cite{LPW}. Furthermore, the HRT conjecture is related to the analytic zero divisor conjecture of Linnell restricted to the Heisenberg group, see \cite{L2, Ros2} for details.

In this paper, we shift focus to finding classes ${\mathcal {F}}$ of measurable functions such that 
\[
\mathcal G(f,\R^2)=\{M_aT_b f:a,b\in \br\}
\]
is linearly independent for all $0 \ne f\in {\mathcal {F}}$.  A specific conjecture of this type was made in \cite{D} (following a question asked in \cite{Heil}), where it was conjectured that if $f$ is in the Schwartz class, then it has linearly independent time-frequency translates.  The main purpose of this paper is to provide a sufficient condition on the decay of $f$ which guarantees the linear independence of its time-frequency translates.  Our main result yields the linear independence of time-frequency translates of $f$ with sufficiently fast one-sided decay. 

\begin{theorem}\label{decay} Suppose $f: \R \to \C$ is a non-zero Lebesgue measurable function such that for all $c>0$, 
\begin{equation}\label{fel}
\lim_{x\to \infty} |f(x)| e^{cx \log x} = 0.
\end{equation}
Then, the set $\mathcal G(f,\R^2)$ of time-frequency translates of $f$ is linearly independent.
\end{theorem}

There are two main ingredients used in proving Theorem \ref{decay}. The first is the Tur\'an-Nazarov inequality \cite{N} which provides a precise estimate on the size of a set where a trigonometric polynomial is small. In Section \ref{S2} we illustrate how this result is used to deduce the linear independence of $\mathcal G(f,\R^2)$ when $f$ has faster than Gaussian decay. The second key ingredient is a technical sufficient condition for the linear independence of time-frequency translates of $f$ which is established using a determinant argument in Theorem \ref{tumb}. This theorem becomes a foundation on which our subsequent results for functions $f$ with faster than exponential decay are shown. In Theorems \ref{grz} and \ref{qm} we show the linear independence of $\mathcal G(f,\Lambda)$ when $\Lambda=\R \times b\Z$ or $f$ satisfies some weak monotonicity condition, respectively. Finally, we derive a  delicate lower bound estimate on products of trigonometric polynomials which is then used to deduce Theorem \ref{decay}.

\section{Gaussian Decay Condition}\label{S2}

In this section we prove that the time-frequency translates of a function with faster than Gaussian decay are linearly independent.  That is, we will show that if $f \ne 0$ satisfies 
\begin{equation}\label{fga}
\lim_{x\to \infty} f(x) e^{cx^2} = 0\qquad\text{for all }c>0,
\end{equation}
then $f$ has linearly independent time-frequency translates. While this result is superseded by Theorem \ref{decay}, our goal here is to give the simplest proof (known to us) of the linear independence of $\mathcal G(f,\R^2)$ provided we are willing to assume sufficient decay on $f$. Furthermore, the proof of Theorem \ref{gooddecay} illustrates the essential role played by the Tur\'an-Nazarov inequality.

Let us introduce some notation and a reformulation of the HRT conjecture. Suppose that there exists a finite set $\Lambda \subset \R^2$ of parameters for time-frequency translates and non-zero complex coefficients
$\{c_{(a,b)}: (a,b) \in \Lambda\}$ such that 
\[
\sum_{(a,b) \in \Lambda} c_{(a,b)} M_aT_b f = 0.
\]
Let $\Gamma = \{b\in \R: \exists a\in \R {\rm \ such \  that \ } (a,b) \in \Lambda\}$.  For each $b\in \Gamma$, we define a non-zero trigonometric polynomial (for us, trigonometric polynomials need not be periodic) $u_b(x) = \sum_{a\in \Gamma_b} c_{(a,b)} e^{2\pi i a x}$, where $\Gamma_b = \{a\in \R: (a,b) \in \Lambda\}$. Since $\Gamma$ is finite we can list its elements in increasing order $b_1<\ldots<b_n$.  By letting $u_i(x)=u_{b_i}(x)$, the linear dependence condition takes the form
\begin{equation}\label{tumb2}
\sum_{b\in \Gamma} u_i(x)T_bf(x)= \sum_{i=1}^n u_{i}(x) f(x-b_i)=0 \qquad\text{for a.e. }x \in\R.
\end{equation}
With these comments, it is clear that showing linear independence of time-frequency translates of $f$ is equivalent to showing that whenever $u_i$ are non-zero trigonometric polynomials and $b_1 < \cdots < b_n$, we have 
\[
\sum_{i=1}^n u_i(x) f(x - b_i) \not= 0.
\]

In view of the recasting of the problem in the above paragraph, it is apparent that the zeros of trigonometric polynomials may play a key role in the study of the HRT conjecture. 
In our proof of Theorem \ref{gooddecay}, given a trigonometric polynomial $u$, we need an estimate on the most that 
$
\inf_{x\in E} |u(x)|
$
can decrease when $E$ is replaced by a subset of measure half of its size.  Intuitively, the worst case scenario is to center the set $E$ around zeros of the trigonometric polynomial, at least when the measure of $E$ is small.  Since the number and multiplicity of zeros of a trigonometric polynomial are uniformly bounded over all intervals of length $K$ (Lander's lemma, c.f. \cite[Lemma 1.3]{N}), it is believable that we should be able to obtain a bound on how much the trigonometric polynomial can shrink.  The following version of the Tur\'an lemma, which is due to Nazarov \cite[Theorem I]{N}, shows that this is indeed the case.

\begin{theorem}[Tur\'an-Nazarov inequality] \label{NazTur}
Let
$u(x) = \sum_{j=1}^m c_j e^{a_j x}$, 
where $a_j, c_j\in {\mathbb C}$, be an exponential polynomial of order
$m$. Let $I\subset\br$ be an interval, and
$E$ a measurable subset of $I$ of positive measure.
Then,
\[
\sup_{x\in I}|u(x)|\le e^{|I|\max|\Re a_j|}
 \biggl\{\frac{A |I|}{|E|}\biggr\}^{m-1}
 \sup_{x\in E}|u(x)|.
\]
Here, $A$ denotes an absolute constant, and $|E|$ is the Lebesgue measure of $E$.
\end{theorem}

As a consequence of Theorem \ref{NazTur} we have the following result about trigonometric polynomials.

\begin{proposition}\label{NazInf}  Let $u(x) = \sum_{j=1}^m c_j e^{2\pi i a_j x}$ be a non-zero trigonometric polynomial, where $a_j\in \br$, $c_j \in \bc$.  Let $K > 0$. Then there exists a constant $C>0$, depending only on $u$ and $K$, such that following are true:
\begin{enumerate}
\item
for all $a\in \br$,
\begin{equation}\label{NazInf1}
\sup_{x\in [a, a + K]} |u(x)| \ge C,
\end{equation}
\item
for any set $E\subset \br$ of positive measure with diameter less than $K$, there exists an $F\subset E$ with $|F| = |E|/2$ and such that 
\begin{equation}\label{NazInf2}
|u(x)| \ge C |E|^{m-1}
\qquad\text{ for all } x\in F.
\end{equation}
\end{enumerate}
\end{proposition}

\begin{proof}
Let $0 < \delta < \min\{|a_j - a_k|: j\not= k\}$.  Suppose that $K > 1/\delta$.  The Montgomery-Vaughan inequality \cite[Theorem 1, Chapter 7]{Mon} states that for any $d_j \in \bc$ we have 
\[
\bigl(K - 1/\delta\bigr) \sum_{j=1}^n |d_j|^2 \le \int_0^K \biggl| \sum_{j = 1}^n d_j e^{2\pi i a_j t} \biggr|^2\ dt
\le \bigl(K + 1/\delta\bigr) \sum_{j=1}^n |d_j|^2.
\]
Therefore, for each $a\in \br$,
\[
\int_a^{a+K} \biggl| \sum_{j = 1}^n c_j e^{2\pi i a_j t} \biggr|^2\,
dt  = \int_0^{K} \biggl| \sum_{j = 1}^n c_j e^{2\pi i a_j a} e^{2\pi i a_j t} \biggr|^2\,
dt\ge \bigl(K - 1/\delta\bigr) \sum_{j=1}^n |c_j|^2.
\]
Thus, the average value of $|u|^2$ on $[a, a + K]$ is at least $(1-1/(K\delta)) \sum_{j=1}^n |c_j|^2$, which shows \eqref{NazInf1} for $K > 1/\delta$.   The case when $K\le 1/\delta$ follows from Theorem \ref{NazTur} and \eqref{NazInf1} for $K > 1/\delta$.  

Given $E \subset \R$ of diameter less than $K$, for $t > 0$ we define $E_t = \{x\in E: |u(x)| \le t\}$.  Choose $t > 0$ such that $|E_t| = |E|/2$.   By Theorem \ref{NazTur}, 
\[
\inf_{x\in E\setminus E_t} |u(x)| \ge t \ge \sup_{x\in E_t} |u(x)| \ge \frac{C}{(2AK)^{m-1}} |E|^{m-1}.
\]
Thus, $F=E\setminus E_t$ satisfies (\ref{NazInf2}).
\end{proof}

\begin{theorem}\label{gooddecay} Let $b_1$ be the smallest element of $B = \{b_1,\ldots,b_n\} \subset \R$.  Suppose that $u_1$ is a non-zero trigonometric polynomial, and that $u_2,\ldots,u_n$ are bounded.  If $f$ is a non-zero measurable function such that 
\[
\sum_{j=1}^n u_j(x) f(x + b_j) = 0 \qquad\text{for a.e. }x\in \R,
\]
then there exists a constant $c > 0$ such that $\limsup_{x\to \infty} f(x) e^{cx^2} > 0$.  In particular, if \eqref{fga} holds, then $f \ne 0$ has linearly independent time-frequency translates.
\end{theorem}

\begin{proof}  Without loss of generality we can assume that $0 = b_1 < b_2 < \cdots < b_n$ and that  $||u_j||_\infty \le 1$ for $j=1,\ldots,n$. Suppose that a trigonometric polynomial $u_1$ is of order $m$.
We claim that there exist constants $\ve, \lambda>0$ and a sequence of measurable sets $\{X_k\}_{k\in\N} \subset \R $ such that for each $k\in \N$ we have:
\begin{enumerate}
\item  $|f(x)| \ge \ve^k (2n)^{-(m-1)(k-1)k/2}$ for all $x\in X_k$,
\item $|X_k| = \lambda/(2n)^{k-1}$, and
\item there exists $j=j(k) \in [2,n]$ such that $X_{k+1} \subset X_{k} + b_j$.
\end{enumerate}

First, we choose $\ve,\lambda>0$ and a set $X_1 \subset \R$ of finite diameter with positive measure $\lambda$ such that $|f(x)| > \ve$ on $X_1$. Next, suppose that $X_1,\ldots,X_k$ have been already defined. In order to define $X_{k+1}$ we use \eone to find $V\subset X_k$ such that $|V| = |X_k|/2$ and
\[
|u_1(x)| \ge C |X_k|^{m-1} = C \bigg(\frac{\lambda}{(2n)^{k-1}}\biggr)^{m-1}\qquad\text{for } x\in V.
\]
Then, for $x\in V$, 
\begin{eqnarray*}
\sum_{j=2}^n \bigl| u_j(x) f(x + b_j)\bigr| \ge
|u_1(x)| |f(x)|
\ge 
\ve^k (2n)^{-(m-1)(k-1)k/2}   C \bigg(\frac{\lambda}{(2n)^{k-1}}\bigg)^{m-1}.
\end{eqnarray*}
By Lemma \ref{easylemma},  there is a $j=j(k)\in [2,n]$ and a set $W \subset V$ of measure $|W|=|V|/n = \lambda/(2n)^{k}$ such that for $x\in W$,
\begin{align*}
|f(x+b_j)| \ge |u_{j}(x) f(x + b_{j}) | &\ge n^{-1} \ve^k (2n)^{-(m-1)(k-1)k/2}   C \bigg(\frac{\lambda}{(2n)^{k-1}}\bigg)^{m-1}
\\&
\ge C\lambda^{m-1} \ve^k (2n)^{-(m-1)k(k+1)/2}.
\end{align*} 
Thus, $X_{k+1} = W + b_j$ satisfies the required properties, provided we choose $\ve< C \lambda^{m-1}$.

By (3) the sets $X_k$ satisfy $X_k \subset (a+kb_2,\infty)$ for some $a\in \R$. Combining this with (1) shows the existence of a constant $c$ such that $f(x)\ge e^{-cx^2}$ for $x\in X_k$. This completes the proof of Theorem \ref{gooddecay}.
\end{proof}

\begin{lemma}\label{easylemma}  If $f_j \ge 0$ for all $j=1,\ldots,n$ and $\sum_{j=1}^n f_j(x) \ge M$ on a set $X$ of measure $\alpha$, then there is a subset $X' \subset X$ of measure $\alpha/n$ and $j$ such that $f_j(x) \ge M/n$ on $X'$.
\end{lemma}

\begin{proof}  Let $X_j = \{x\in X: f_j(x) \ge M/n\}$.  Clearly, $\cup X_j = X$, so at least one of the $X_j$'s must have measure at least $\alpha / n$.
\end{proof}

\section{Exponential Decay Condition}\label{S3}

In the remainder of this paper, we prove the linear independence of time-frequency translates for functions with faster than exponential decay \eqref{fex} which satisfy some additional mild conditions. In particular, we shall establish our main linear independence result, Theorem \ref{decay}, for functions satisfying the decay condition \eqref{fel}.\footnote{After giving a talk on a preliminary version of this paper, we were informed of unpublished work by Benedetto and Bourouihiya \cite{BB}, who study the linear independence of time-frequency translates for functions with certain behaviors at infinity.  Their result most closely connected to the results in this section is that if $f$ and $f^\prime$ satisfy $e^{tx}f(x), e^{tx}f^\prime(x) \in L^2$ for all $t>0$, and $|f|$ is ultimately decreasing, then $f$ has linearly independent time-frequency shifts.}

We start by introducing a technical sufficient condition \eqref{tumb1} for the linear independence of time-frequency translates of a measurable function. This in turn becomes a foundation on which subsequent linear independence results in this section are shown. Define the space of all Lebesgue measurable on the real line by
\[
\mathcal M=\{f: \R \to \C \text{ is Lebesgue measurable}\}.
\]
As it is customary, we shall identify functions in $\mathcal M$ which are equal almost everywhere.

\begin{theorem}\label{tumb}
Let $f\in \mathcal M$. Suppose that for any non-zero trigonometric polynomial $u$, any finite subset $B=\{b_1,\ldots,b_n\} \subset \R_+$, and any $M>0$, the set
\begin{equation}\label{tumb1}
E=E_{u,M,B} = \bigg\{ x\in \R_+: |u(x)f(x)|> M \sum_{i=1}^n |f(x+b_i)|
\bigg\}
\end{equation}
has positive measure. Then, $\mathcal G(f,\R^2)$ is linearly independent. 

Moreover, if \eqref{tumb1} is relaxed to hold only for finite subsets $B \subset b\N$, where $b>0$, then $\mathcal G(f,\R\times b\Z)$ is linearly independent.
\end{theorem}

\begin{proof}
Suppose for the sake of contradiction that there exist real numbers $b_1<\ldots<b_N$ and trigonometric polynomials $u_1,\ldots,u_N$ such that 
\begin{equation*}
\sum_{i=1}^N u_{i}(x) f(x-b_i)=0 \qquad\text{for a.e. }x \in\R.
\end{equation*}
We shall prove that our hypothesis \eqref{tumb1} implies that there exist sets of positive measure $Q_1,\ldots, Q_N \subset \R$ such that the matrix 
\[
M_N = \begin{pmatrix}
u_{1}(x_1) f(x_1 - b_{1}) &  u_{ 2}(x_1) f(x_1 -b_{2}) & \cdots & u_N(x_1) f(x_1 -b_N)\\
u_{1}(x_2) f(x_2 -b_{1}) & u_{2}(x_2) f(x_2 -b_{ 2}) &\cdots & u_N(x_2) f(x_2 -b_N)\\
 \vdots & & &\vdots\\
 u_{1}(x_N)  f(x_N - b_{1}) & u_{2}(x_N)  f(x_N - b_{2}) & \cdots & u_N(x_N) f(x_N  -b_N)
\end{pmatrix}
\]
has non-zero determinant for almost all $(x_1,\ldots,x_N) \in Q_1 \times \cdots \times Q_N$. This contradicts our hypothesis  that the sum of the rows of $M$ are zero almost everywhere.

For each $1\le n\le N$, and $(x_1,\ldots,x_n)\in \R^n$, we consider the principal $n\times n$ submatrix of $M_N$ given by 
\[
M_n = M_n(x_1,\ldots,x_n) =
\begin{pmatrix}
u_{1}(x_1) f(x_1 -b_1)  & \cdots & u_n(x_1) f(x_1 - b_n)\\
   \vdots &  &\vdots\\
 u_{n}(x_n)  f(x_n - b_1) & \cdots & u_n(x_n) f(x_n  - b_n)
\end{pmatrix}.
\]
We will show by induction the existence of sets of positive measure $Q_1,\ldots,Q_n$ and positive constants $c_1,\ldots,c_n$ and  $\delta_1,\ldots,\delta_n$ such that
\begin{align}\label{tumb3}
|f(x - b_j)| &\le c_n \qquad\text{for a.e. }x \in \bigcup_{i=1}^n Q_i,\ j=1, \ldots, N,
\\
\label{tumb4}
|\det M_n(x_1,\ldots,x_n)| &\ge \delta_n
\qquad\text{for a.e. }
(x_1,\ldots,x_n) \in Q_1 \times \cdots \times Q_n.
\end{align}

The base case $n=1$ follows trivially from the presence of strict inequality in \eqref{tumb1}. Suppose that \eqref{tumb4} holds for some $1\le n<N$. Let $U$ be an upper bound for the trigonometric polynomials $u_i$; that is, $|u_i(x)| \le U$ for all $1\le i  \le N$ and for all $x$.  Let $\Sigma$ be the set of all permutations of $\{1,\ldots,n+1\}$ such that $\sigma(n+1) \ne n+1$. Then, for any $(x_1,\ldots,x_n,x_{n+1}) \in Q_1 \times \ldots \times Q_n \times \R$,
\begin{equation}\label{tumb6}
\begin{aligned}
&| \det M_{n+1}(x_1,\ldots,x_n,x_{n+1}) | \\
&\ge | u_{n + 1}(x_{n+1}) f(x_{n+1} - b_{n + 1}) \det M_{n}(x_1,\ldots,x_n)| - \bigg| \sum_{\sigma \in \Sigma} \prod_{k=1}^{n+1} u_{\sigma(k)}(x_k) f(x_k - b_{\sigma(k)}) \bigg| \\
& \ge \delta_n | u_{n + 1}(x_{n+1}) f(x_{n+1} - b_{n + 1})|-
n! U^{n+1} (c_n)^n \sum_{i=1}^n |f(x_{n+1} - b_i)|.
\end{aligned}
\end{equation}
The last estimate is a consequence of breaking the sum over $\sigma \in \Sigma$ with $\sigma(n+1)=i$, where $1\le i \le n$.
By our hypothesis \eqref{tumb1}, the set 
\[
E=\bigg\{ x_{n+1} \in \R_+:  | u_{n + 1}(x_{n+1}+b_{n+1} ) f(x_{n+1} )|
>
M \sum_{i=1}^n |f(x_{n+1} +(b_{n+1}- b_i))|\bigg\},
\]
where $M=2 n! U^{n+1} (c_n)^n /\delta_n$, has positive measure. We momentarily set $Q_{n+1}=b_{n+1}+E$. Then, by \eqref{tumb6} we have that for almost every $(x_1,\ldots,x_{n+1}) \in Q_1 \times \ldots \times Q_{n+1}$,
\[
| \det M_{n+1}(x_1,\ldots,x_{n+1}) | \ge \frac{\delta_n}{2} | u_{n + 1}(x_{n+1}) f(x_{n+1} - b_{n + 1})|>0.
\] 
Thus, by restricting to a (positive measure) subset of $Q_{n+1}$ if necessary, we can find two  constants $c_{n+1}, \delta_{n+1}>0$ such that \eqref{tumb3} and \eqref{tumb4} hold, as desired. This completes the proof of the first part of Theorem \ref{tumb}. The ``moreover'' part follows immediately by requiring in \eqref{tumb2} that each $b_i \in b\Z$.
\end{proof}

We say that $f\in \mathcal M$ has {\it faster than exponential decay} if for all $c>0$, 
\begin{equation}\label{fex}
\lim_{x\to \infty} |f(x)| e^{cx} = 0.
\end{equation}
As the first application of Theorem \ref{tumb}, we shall give a proof of the linear independence of the Gabor system $\mathcal G(f,\Lambda)$ for functions $f \ne 0$ with faster than exponential decay when $\Lambda= \br \times b\Z$. In order to establish Theorem \ref{grz} we will need the following lemma about products of trigonometric polynomials, which is a consequence of the Tur\'an-Nazarov inequality.

\begin{lemma}\label{prod}
Let $u$ be a non-zero trigonometric polynomial and let $K>0$. There exists a constant $\eta=\eta(u,K)$ such that for any finite subset $B=\{b_1,\ldots,b_k\} \subset \R$ and for any interval $I$ of length $K$, there exists a measurable subset $E \subset I$ with $|E|>|I|/2$
such that
\begin{equation}\label{prod1}
\inf_{x\in E} \bigg| \prod_{i=1}^k u(x+b_i) \bigg| \ge e^{-\eta k}.
\end{equation}
\end{lemma}


\begin{proof}
Let
$u(x) = \sum_{j=1}^m c_j e^{2\pi i a_j x}$, 
where $c_j\in {\mathbb C}$, $a_j\in \R$, be a non-zero trigonometric polynomial of order $m$.
Without loss of generality we can assume that $||u||_\infty \le 1$. Indeed, replacing $u$ by $u/||u||_\infty$ in \eqref{prod1} leads to an additional factor $||u||_\infty^k$ which is absorbed by the right hand side of \eqref{prod1}.

Let $I$ be an interval of length $K$. For any $b\in \R$ and $t>0$ we define 
\begin{equation}\label{prod1a}
E_{b}(t)= \{x\in I: |u(x+b)|< t\}.
\end{equation} 
By Theorem \ref{NazTur} and Proposition \ref{NazInf}, we have
\[
\sup_{x\in E_b(t)} |u(x+b)| = t \ge C (AK)^{1-m} |E_b(t)|^{m-1}.
\]
Thus,
\begin{equation}\label{prod2}
|E_b(t)| \le C t^{1/(m-1)} \qquad\text{for all } t>0.
\end{equation}
Hence,
\begin{multline*}
\int_I \log_- |u(x+b)| dx = \int_0^1 |\{x\in I: |u(x+b)|<t\}|\frac{dt}{t} = \int_0^1 |E_b(t)| \frac{d t}{t} 
\\
\le C \int_0^1 t^{-1+1/(m-1)} dt =C(m-1)<\infty,
\end{multline*}
where $\log_-t=\max(0,-\log t)$ is the negative part of $\log$.
Since $||u||_\infty \le 1$, we have
\begin{equation}\label{prod4}
\int_I
\log_-\bigg|\prod_{i=1}^k u(x+b_i) \bigg| dx = \sum_{i=1}^k 
\int_I \log_-|u(x+b_i)| dx \le C(m-1)k.
\end{equation}
For $t>0$ define
\[
E(t) =\bigg\{x\in I: \log_- \bigg|\prod_{i=1}^k u(x+b_i)\bigg|>t \bigg\}.
\]
By Chebyshev's inequality and \eqref{prod4} we have
\[
|E(t)|=\bigg|\bigg\{x\in I: \bigg|\prod_{i=1}^k u(x+b_i) \bigg| <e^{-t} \bigg\}\bigg| 
\le C(m-1)k/t \qquad\text{for }t>0.
\]
By setting $t= 3 C(m-1)k/|I|$, we have $|E(t)| \le 1/3|I|$. Set $E=I \setminus E(t)$. Then, $|E| \ge 2/3|I|$ and
\[
\bigg| \prod_{i=1}^k u(x+b_i) \bigg| \ge e^{-\eta k} \qquad\text{for } x\in E,
\]
where $\eta= 3C(m-1)/|I|$. This completes the proof of Lemma \ref{prod}.
\end{proof}

\begin{theorem}\label{grz}
Suppose that $0 \ne f\in \mathcal M$ has faster than exponential decay. Then, $\mathcal G(f,\Lambda)$ is linearly independent for $\Lambda = \R \times b\Z$, where $b>0$.
\end{theorem}

\begin{proof}
By Theorem \ref{tumb} it suffices to show that for any trigonometric polynomial $u \ne 0$, $n\in \N$, and $M>0$,
\[
E =  \bigg\{ x\in \R_+: |u(x)f(x)|> M \sum_{i=1}^n |f(x+ib)|
\bigg\}
\]
has positive measure. On the contrary, suppose that there exists some choice of $u$, $n$, and $M>0$ such that
\begin{equation}\label{grz2}
|u(x)f(x)| \le M \sum_{i=1}^n |f(x+ib)|
\qquad\text{for a.e. } x\in \R_+.
\end{equation}
By induction this implies that for any $k\in\N$, we have
\begin{equation}\label{grz4}
|f(x)|\prod_{j=0}^k|u(x+jb)| \le M (M+||u||_\infty)^k \sum_{i=1}^n |f(x+(k+i)b)|
\qquad\text{for a.e. } x\in \R_+.
\end{equation}
Indeed, the base case $k=0$ corresponds to \eqref{grz2}. Assuming that \eqref{grz4} holds for $k-1$, we have
\begin{align*}
|f(x)|&\prod_{j=0}^k|u(x+jb)| 
\\
&\le 
M (M+||u||_\infty)^{k-1} \bigg( |u(x+kb)f(x+kb)|+ |u(x+kb)|\sum_{i=2}^n |f(x+((k-1)+i)b)| \bigg)
\\
&\le 
M (M+||u||_\infty)^{k-1} \bigg( M \sum_{i=1}^n |f(x+(k+i)b)| + ||u||_\infty \sum_{i=1}^{n-1} |f(x+((k+i)b)| \bigg).
\end{align*}
Thus, \eqref{grz4} holds for all $k\in\N$. It remains to see that \eqref{grz4} can not hold for functions $f$ with faster than exponential decay.

By the Lebesgue differentiability theorem applied to the set $\{x\in \R: |f(x)|>\ve\}$, there exists an interval $I \subset \R$ and a constant $\ve>0$ such that
\begin{equation}\label{grz6}
|\{x\in I: |f(x)|>\ve \}| > |I|/2.
\end{equation}
Moreover, we can assume that $I \subset \R_+$ since it is well known that we have linear independence of $\mathcal G(f,\R^2)$ for functions supported in a half-line. Combining \eqref{grz4} and \eqref{grz6} with Lemma \ref{prod} yields for each $k\in\N$, a set $E_k \subset I$ of positive measure such that
\[
M(M+||u||_\infty)^k  \sum_{i=1}^n |f(x+(k+i)b)| >\ve e^{-\eta(k+1)}
\qquad
\text{for }x\in E_k.
\]
By Lemma \ref{easylemma}, there exists $i=i(k)\in [1,n]$ such that
\[
|f(x+(k+i)b)| >  \frac{\ve}{e^\eta Mn} (M+||u||_\infty)^{-k} e^{-\eta k}
\]
for $x$ in some subset of $I$ with positive measure. This contradicts 
our hypothesis that $f$ has faster than exponential decay and completes the proof of Theorem \ref{grz}.
\end{proof}

As the second application of Theorem \ref{tumb}, we shall deduce the linear independence of the Gabor system $\mathcal G(f,\R^2)$ for functions with faster than exponential decay which are quasi-monotone. We say that $f\in \mathcal M$ is {quasi-monotone} if
\begin{equation}\label{qm0}
\forall b>0 \ \exists C=C(b)>0 \ \forall x >0 \qquad  |f(x+b)| < C |f(x)|.
\end{equation}

\begin{theorem}\label{qm}
Suppose that $0 \ne f\in \mathcal M$ has faster than exponential decay and $f$ is quasi-monotone. Then, $\mathcal G(f,\R^2)$ is linearly independent.
\end{theorem}

\begin{proof}
Let $u$, $M$, and $B$ be as in Theorem \ref{tumb}. Define the constant 
\[
\tilde M = M \bigg(1+ \sum_{i=2}^n C(b_i-b_1) \bigg),
\]
where $b_1$ is the smallest element of $B=\{b_1,\ldots,b_n\} \subset \R_+$, and $C(b)$ is the same as in \eqref{qm0}.
The quasi-monotonicity of $f$ implies that the set $E_{u,M,B}$ in \eqref{tumb1} satisfies
\begin{equation}\label{qm2}
E_{u,\tilde M, b_1} = \bigg\{ x\in \R_+: |u(x)f(x)|> \tilde M |f(x+b_1)| \bigg\} \subset E_{u,M,B}.
\end{equation}
The proof of Theorem \ref{grz} shows that the set on the left-hand side of \eqref{qm2} has positive measure. Consequently, $E_{u,M,B}$ also has positive measure. Hence, we reach the required conclusion by applying Theorem \ref{tumb}.
\end{proof}

As the third application of Theorem \ref{tumb}, we shall give the proof Theorem \ref{decay}. That is, we establish the linear independence of the Gabor system $\mathcal G(f,\R^2)$ for functions with slightly faster decay than \eqref{fex}, but without any extra restrictions such as those in Theorems \ref{grz} and \ref{qm}. In order to establish this result, we will need the following more elaborate variant of Lemma \ref{prod} about products of trigonometric polynomials.

\begin{lemma}\label{product}
Let $u$ be a non-zero trigonometric polynomial, let $B=\{b_1,\ldots,b_n\} \subset \R$ be a finite set of real numbers, and let $K>0$. Then, there exists a constant $\eta=\eta(u,n,K)>0$ such that for any interval $I$ of length $K$ and any $k\in\N$, there exists a measurable subset $E \subset I$ with $|E|>|I|/2$
such that
\begin{equation}\label{product1}
\sup_{x\in E} \sum_{i(1)=1}^n \ldots \sum_{i(k)=1}^n \bigg| \prod_{j=1}^k u\bigg(x+\sum_{l=1}^j b_{i(l)}\bigg)\bigg|^{-1} \le e^{\eta k \log k}.
\end{equation}
\end{lemma}

\begin{proof}
Without loss of generality we can assume that $||u||_\infty \le 1$ as in Lemma \ref{prod}. Let $I$ be an interval of length $K$. For fixed $k\in\N$ define the set
\[
\Sigma=\bigg\{ \sum_{i=1}^n\alpha_i b_i: \sum_{i=1}^n \alpha_i \le k,\ \alpha_i\in\N_0 \bigg\}.
\]
Since the sequence $(\alpha_1, \ldots, \alpha_n, k-(\alpha_1+\ldots+\alpha_n))$ represents a partition of $k$ into $n+1$ blocks, we have
\begin{equation}\label{product2}
\# |\Sigma| \le \binom{k+n}{k} \le C k^{n}.
\end{equation}
For any subset $\sigma=\{\sigma(1),\ldots,\sigma(k)\} \subset \Sigma$ of size $k$, define the function
\[
f_\sigma(x) = \prod_{i=1}^k \frac{1}{|u(x+\sigma(i))|}.
\]

Let $t>0$. Suppose that for some $x\in I$ we have
\begin{equation}\label{product4}
\sum_{i(1)=1}^n \ldots \sum_{i(k)=1}^n \bigg| \prod_{j=1}^k u\bigg(x+\sum_{l=1}^j b_{i(l)}\bigg)\bigg|^{-1} >t.
\end{equation}
This implies that there exists a subset $\sigma\subset \Sigma$ of size $k$ such that $f_\sigma(x)>t/n^k$. Since $f_\sigma$ is a product of $k$ functions, at least one of them must take value greater than $(t/n^k)^{1/k}$. That is,
\[
x \in \bigcup_{i=1}^n E_{\sigma(i)} \bigg( \frac{n}{t^{1/k}} \bigg) 
\subset \bigcup_{b\in \Sigma}  E_b \bigg( \frac{n}{t^{1/k}} \bigg),
\]
where $E_b(t)$ is given by \eqref{prod1a}. Thus, using \eqref{prod2} and \eqref{product2}, the Lebesgue measure of the set of points $x\in I$ satisfying \eqref{product4} is bounded by
\[
\bigg|  \bigcup_{b\in \Sigma}  E_b \bigg( \frac{n}{t^{1/k}} \bigg) \bigg| \le \#|\Sigma| \max_{b\in \Sigma} \bigg|E_b\bigg( \frac{n}{t^{1/k}} \bigg) \bigg| \le C k^n \frac{n^{1/(m-1)}}{t^{1/k(m-1)}} \le C' k^n t^{-\frac{1}{k(m-1)}}.
\]
If we wish that the measure of this set does not exceed $K/2$, we are led to the inequality
\[
t> \bigg(\frac{2C'}{K} \bigg)^{(m-1)k} k^{n(m-1)k}.
\]
Thus, there exists a constant $\eta>0$, which is independent of the choice of $k\in\N$, such that $t=e^{\eta k \log k}$ satisfies the above bound. Consequently, the set $E$ of points $x\in I$ such that the inequality \eqref{product4} {\bf fails} has measure at least $K/2$. This completes the proof of Lemma \ref{product}.
\end{proof}

We are now ready to prove Theorem \ref{decay}.

\begin{proof}[Proof of Theorem \ref{decay}]
By Theorem \ref{tumb} it suffices to show that for any trigonometric polynomial $u \ne 0$, any finite subset $B=\{b_1,\ldots,b_n\} \subset \R_+$ and any $M>0$, the set $E_{u,M,B}$ given by \eqref{tumb1} has positive measure.
On the contrary, suppose that for some choice of $u$, $B$, and $M>0$ we have
\begin{equation}\label{fexp2}
|f(x)| \le M \sum_{i=1}^n \frac{|f(x+b_i)|}{|u(x)|}
\qquad\text{for a.e. } x\in \R_+.
\end{equation}
By recursion, \eqref{fexp2} implies that
\begin{equation}\label{fexp4}
|f(x)| \le M^k \sum_{i(1)=1}^n \ldots \sum_{i(k)=1}^n 
\bigg|f\bigg(x+\sum_{l=1}^k b_{i(l)}\bigg)\bigg| 
  \prod_{j=1}^k \bigg| u\bigg(x+\sum_{l=1}^{j-1} b_{i(l)}\bigg)\bigg|^{-1}.
\end{equation}

By the Lebesgue differentiability theorem applied to the set $\{x\in \R: |f(x)|>\ve\}$, there exists an interval $I \subset \R$ and a constant $\ve>0$ such that
\begin{equation}\label{fexp6}
|\{x\in I: |f(x)|>\ve \}| > |I|/2.
\end{equation}
Moreover, we can assume that $I=[a,a+K] \subset \R_+$ since we have linear independence of $\mathcal G(f,\R^2)$ for functions supported in a half-line. Also, we can assume that $b_1$ is the smallest element of $B$. Then, by Lemma \ref{product}, for any $k\in\N$, there exists a subset $E_k \subset I$ with $|E_k|>|I|/2$ such that
\begin{equation}\label{fexp8}
\begin{aligned}
|f(x)| & \le M^k \sup_{y\ge a+kb_1} |f(y)|
 \sum_{i(1)=1}^n \ldots \sum_{i(k)=1}^n 
\bigg| \prod_{j=1}^k u\bigg(x+\sum_{l=1}^{j-1} b_{i(l)}\bigg)\bigg|^{-1} 
\\
& \le M^k e^{\eta k \log k} \sup_{y\ge a+kb_1} |f(y)|
\qquad\text{for }x\in E_k. 
\end{aligned}
\end{equation}
Combining \eqref{fexp6} and \eqref{fexp8}, we have that
\[
\sup_{y\ge a+kb_1} |f(y)| \ge  \ve M^{-k} e^{-\eta k \log k}
\qquad\text{for all }k\in\N.
\]
This contradicts our decay hypothesis \eqref{fel} and completes the proof of Theorem \ref{decay}.
\end{proof}

The following result, which is an immediate corollary of Theorem \ref{decay}, was first shown by Heil \cite{Heil} in the case when $\epsilon = 1$. 

\begin{corollary} Let $p(x)=a_m x^m+\ldots +a_0$ be an algebraic polynomial and let $\epsilon > 0$.  Then, the function $p(x) e^{-|x|^{1+\epsilon}}$ has linearly independent time-frequency translates.
\end{corollary}

Observe that time-frequency translates of $f$ are linearly independent if and only if time-frequency translates of its Fourier transform $\hat f$ are linearly independent. Thus, the results in this section also give conditions on $\hat f$ (and in the case of Theorem \ref{grz}, on $\Lambda$) which guarantee the linear independence of time-frequency translates of $f$. We leave the details to the reader.


\section{Conclusions and final remarks}

We end this paper by posing questions which are closely related to the HRT conjecture. These questions seem to be just on the edge of what seems to be doable. Note that in this paper we have considered only one-sided decay conditions so far. That is, decay as $x\to \infty$ or by the symmetry as $x\to -\infty$.  As such, there is a natural boundary of $e^{-cx}$ at which our results fail to hold, since the translates of $e^{-cx}$ are linearly dependent.  So, functions with faster than exponential decay are a very natural class to consider. 

\begin{question} If $f \ne 0$ satisfies faster than exponential decay \eqref{fex}, then does $f$ have linearly independent time-frequency translates?
\end{question}

The results of Section \ref{S3} strongly suggest a positive answer to this question. In fact, Theorems \ref{decay}, \ref{grz}, or \ref{qm} show linear independence if we assume additional decay on $f$, a restriction on time-frequency translates $\Lambda$, or some weak monotonicity on $f$, respectively. However, the general case when $f$ merely satisfies \eqref{fex} remains open.

Another natural area worth investigating is to impose two-sided decay conditions. While linear independence fails for functions with one sided exponential decay such as $e^{-x}$, it is natural to wonder what happens for functions whose tails are exactly equal to $e^{-|x|}$.  The following theorem indicates that we have linear independence of time-frequency translates under such a scenario. As an immediate consequence of Theorem \ref{e2x} we have that $f(x)=e^{-|x|}$ has linearly independent time-frequency translates.

\begin{theorem}\label{e2x}
Let $f\in\mathcal M$ be such that there exist constants $a,c \in\bc$ and $b\in\br$ satisfying
\begin{equation}\label{bo}
f(x) = c e^{ax} \qquad\text{for a.e. }x>b, 
\end{equation}
or, alternatively, for a.e. $x<b$.
Then, $\mathcal G(f,\R^2)$ is linearly independent, unless $f(x)=c e^{ax}$ for a.e. $x\in\br$.
\end{theorem}

\begin{proof}
Let $b_0\in\br$ be the minimal $b$ such that \eqref{bo} holds. That is, for every $\delta>0$ there exists a set $E \subset (b_0-\delta,b_0)$ of positive measure such that
\begin{equation}\label{bn}
f(y) \ne c e^{ay} \qquad\text{for }y\in E.
\end{equation}
On the contrary, suppose that $\mathcal G(f)$ is linearly dependent. Then, there exist $x_1,\ldots, x_n \in \br$ and  non-zero trigonometric polynomials $u_1, \ldots, u_n$ such that
\[
u_1(x)f(x-x_1)+ \ldots + u_n(x) f(x-x_n)=0 
\qquad\text{for a.e. }x\in\br.
\]
By reordering we can assume that $x_1>x_2>\ldots>x_n$. Therefore, we have
\[
\begin{aligned}
u_1(x) c e^{a(x-x_1)} + u_2(x) c e^{a(x-x_2)} + \ldots + u_n(x) c e^{a(x-x_n)} &=0 \qquad\text{for a.e. } x>b_0+x_1,\\
u_1(y) f(y-x_1) + u_2(y) c e^{a(y-x_2)} + \ldots + u_n(y) c e^{a(y-x_n)} &=0 \qquad\text{for a.e. } b_0+x_2<y<b_0+x_1.
\end{aligned}
\]
The first equation implies that $u_1(x) e^{-ax_1} + \ldots + u_n(x) e^{-ax_n} =0$ for a.e. $x>b_0+x_1$ and thus for all $x\in\br$. Substituting this into the second equation yields
\[
u_1(y) f(y-x_1) -u_1(y) c e^{a(y-x_1)} =0 
\qquad\text{for a.e. } b_0+x_2<y<b_0+x_1.
\]
Since $u_1$ is a non-zero trigonometric polynomial we have $f(y) = ce^{ay}$ for a.e. $b_0+x_2-x_1<y<b_0$. This contradicts \eqref{bn}. Therefore, $\mathcal G(f,\R^2)$ is linearly independent. The case when \eqref{bo} holds for a.e. $x<b$ follows by symmetry.
\end{proof}

Our second question is motivated in part by an attempt to improve Theorem \ref{qm} by replacing $f$ with $f {\mathbf 1}_K$ for some measurable subset $K\subset \R$. While the decay condition is inherited in this process, the quasi-monotonicity condition \eqref{qm0} is not. Since compactly supported functions have linearly independent time-frequency translates, it is natural to believe that functions whose support is sufficiently sparse must also have the same property.  

\begin{question}\label{q2} Let $K$ be the support of a non-zero function $f\in L^2(\R)$.  If $K$ has lower Beurling density zero, that is, if 
\[
\liminf_{R\to \infty}\inf_{x\in \br} |K\cap [x, x + R]| = 0,
\]
then does $f$ have linearly independent time-frequency translates?  
\end{question}

As a particular example, note that if there are three distinct translations $\{0, b_1, b_2\}$, then a necessary condition for $f$ to have linear dependence \eqref{tumb2} is that
\begin{itemize}
\item $K \subset (K + b_1) \cup (K + b_2)$,
\item $K + b_1 \subset K \cup (K + b_2)$, and
\item $K + b_2 \subset (K + b_1) \cup K$.
\end{itemize}
It would be interesting to determine whether those three conditions already imply that $K$ does not have lower Beurling density zero.

\bibliographystyle{amsplain}

\begin{thebibliography}{99}

\bibitem{BB} J. Benedetto, A. Bourouihiya,
\emph{Linear independence of finite Gabor systems generated by functions with certain behaviors at infinity}, preprint (2012).

\bibitem{BS} M. Bownik and D. Speegle, 
\emph{Linear independence of Parseval wavelets,} 
Illinois J. Math. \textbf{54} (2010), no. 2, 771--785. 

\bibitem{D} C. Demeter, 
\emph{Linear independence of time frequency translates for special configurations,}
Math. Res. Lett. \textbf{17} (2010), no. 4, 761--779. 

\bibitem{DG} C. Demeter and Z. Gautam, \emph{On the finite linear independence of lattice Gabor systems,} Proc. Amer. Math. Soc. (to appear).

\bibitem{DZ} C. Demeter, A. Zaharescu,
\emph{Proof of the HRT conjecture for $(2,2)$ configurations}, preprint (2010).

\bibitem{Heil}
C. Heil, 
{\it Linear independence of finite Gabor systems},
``Harmonic Analysis and Applications", C. Heil, ed., Birkh\" auser, Boston, (2006), 171--206.

\bibitem{HRT}
C. Heil, J. Ramanathan, P. Topiwala, 
{\it Linear independence of time-frequency translates},
Proc. Amer. Math. Soc. {\bf 124} (1996), 2787--2795. 

\bibitem{Ku}
G. Kutyniok,
\emph{Linear independence of time-frequency shifts under a generalized Schr\"odinger representation},
Arch. Math. (Basel) {\bf 78} (2002), no. 2, 135--144. 

\bibitem{LPW}
J. Lawrence, G. Pfander, D. Walnut, 
\emph{Linear independence of Gabor systems in finite dimensional vector spaces}, 
J. Fourier Anal. Appl. {\bf 11} (2005), no. 6, 715--726. 

\bibitem{L2} P. Linnell, 
Analytic versions of the zero divisor conjecture. \emph{Geometry and cohomology in group theory} (Durham, 1994), 209--248,
London Math. Soc. Lecture Note Ser., \textbf{252}, Cambridge Univ. Press, Cambridge, 1998. 

\bibitem{Lin}
P.~Linnell, 
{\it von Neumann algebras and linear independence of translates}, Proc.~Amer.~Math.~Soc.  \textbf{127}  (1999),  3269--3277.


\bibitem{Mon}
H. L. Montgomery, Ten lectures on the interface between analytic number theory and harmonic analysis. CBMS Regional Conference Series in Mathematics, 84. {\sl American Mathematical Society, Providence, RI,} 1994. xiv+220 pp.

\bibitem{N}
F.~Nazarov,
{\it Local estimates for exponential polynomials and their applications to inequalities of the uncertainty principle type,} St. Petersburg Math. J. \textbf{5} (1994), no.~4, 663--717.

\bibitem{Ros2}
J.~Rosenblatt, 
{\it Linear independence of translations}, 
 Int. J. Pure Appl. Math. {\bf 45}  (2008), no.~3, 463--473.

\end{thebibliography}

\end{document}